\documentclass[12pt,a4paper,final]{amsart}
\usepackage{amsmath,amsfonts,amssymb,amsthm}
\usepackage[final]{graphicx}
\usepackage[notref,notcite]{showkeys}
\usepackage{comment}
\usepackage[left=3cm,right=3cm,top=3cm,bottom=3cm]{geometry}
\usepackage{mathtools}
\usepackage{paralist}
\usepackage{fancyhdr}
\usepackage{enumitem}
\usepackage[all]{xy}
\usepackage[ruled,vlined,algo2e,norelsize]{algorithm2e}
\usepackage{tikz}
\usepackage{libertine}
\usepackage{libertinust1math}
\usepackage{dsfont}
\usepackage[T1]{fontenc}
\usepackage{float}
\usepackage{url}
\usepackage{hyperref}
\usepackage[obeyDraft, colorinlistoftodos]{todonotes}
\usepackage[foot]{amsaddr}

\theoremstyle{plain}
\newtheorem{thm}{Theorem}
\newtheorem{prop}[thm]{Proposition}

\theoremstyle{definition}

\newtheorem{rem}[thm]{Remark}
\newtheorem{ex}[thm]{Example}

\newcommand{\N}{\mathbb N}
\newcommand{\Z}{\mathbb Z}
\newcommand{\R}{\mathbb R}
\renewcommand{\S}{\mathbb S}
 
\newcommand{\E}{\mathbb E} 

\newcommand{\F}{\mathbf F} 
\newcommand{\X}{\mathbf X} 
 
\newcommand{\U}{\mathbf U} 
\renewcommand{\P}{\mathbb P}
\renewcommand{\E}{\mathbb E}
\newcommand{\V}{\mathbb V}
\renewcommand{\1}{\mathds 1}
\DeclareMathOperator{\diag}{diag}
\DeclareMathOperator{\ms}{ms}
\DeclareMathOperator{\years}{years}

\title{A risk analysis framework for real-time control systems}
\date{\today}
\author{Mads Rystok Bisgaard}
\author{Lukas Hewing}
\author{Alexander Domahidi}

\begin{document}
\maketitle

\begin{abstract}
    We present a Monte Carlo simulation framework for analysing the risk involved in deploying real-time control systems in safety-critical applications, as well as an algorithm design technique allowing one (in certain situations) to robustify a control algorithm. Both approaches are very general and agnostic to the initial control algorithm. We present examples showing that these techniques can be used to analyse the reliability of implementations of non-linear model predictive control algorithms.
\end{abstract}

\section{Introduction}
A central question in embedded systems theory and practice is that of safety and reliability. As systems for safety-critical applications are being built using "best effort practices" \cite{Sifakis07} to an ever increasing extend this is more true today than ever before.
We model the real-world system, whose state $x \in \X$ is to be controlled, by a discrete dynamical system $f$ and denote our control system by $c$:
\begin{align}
f&:\X \times \U \to \X \label{dynamical_system} \\
c&:\X \to \U. \label{control}
\end{align}
Here we are concerned with the question of how to assess the \emph{reliability} of $c$, or in other words the \emph{risk} involved in deploying $c$. For this purpose, we will assume we can assess the reliability of $c$ at a given $x\in \X$. We formalize this by a Boolean function $\F: \X \to \{0,1\}$, with the interpretation that $c(x)$ is considered good/reliable if $\F(x)=1$ and bad/unreliable if $\F(x)=0$. The ability to determine the quality of $c$ at a given $x\in \X$ seems needed in order to say anything meaningful about the reliability of $c$ as a function.

For real world applications one should think of a hierarchical control system in which $c$ is a medium-level controller which is required to determine a control signal within a given time budget and pass it onto a low-level controller. So in the above notation, one thinks of $f$ as being applied at a given frequency and every time it is applied $c$ must have found a control $u$ such that $f(x,u)$ is an optimal/feasible trajectory. If the evaluation of $c$ can be done multiple times then one can find multiple control candidates $u_1, \ldots, u_{N}$ in between two applications of $f$, and the natural requirement is that \emph{at least one of these} gives rise to an optimal/feasible trajectory. Below we present an approach to modelling such a situation as well as a technique for estimating the probability of failure within this model, i.e. the probability of the event that none of the $u_1,\ldots, u_N$ give rise to an acceptable trajectory.
The main result is a framework within which extremely small failure probabilities are computationally tractable along with a design technique which (in certain situations) improves robustness of a control algorithm.

Importantly, one should think of $c$ as a control system implemented as a piece of software, compiled and deployed on a target platform. It is a crucial feature of the risk analysis presented here, that it can be done on the actual target hardware, as opposed to being only a theoretical guarantee relying on e.g. mathematical/code analysis. Such an analysis could easily fail due to myriads of implementation issues (bugs, floating point errors, numerical differences due to different CPUs, compiler differences etc).

\begin{ex}
At Embotech AG we develop cutting edge decision making software for industrial applications. Our core technology is \href{https://www.embotech.com/products/forcespro/overview/}{FORCESPRO} \cite{Domahidi17, DomahidiJerez14}, a codegenerator for numerical optimization software. Given its required input, a FORCESPRO solver produces (together with a solution to the optimization problem) an exitflag indicating the quality of the solution. Exitflag  1 means the solver successfully found a KKT point (which, given a suitably designed control system, in particular means the corresponding control gives rise to a feasible trajectory). For concreteness one can think of $\F$ as a FORCESPRO solver whose output consists simply of its exitflag.
\end{ex}

\subsubsection*{Related work}
Monte Carlo simulation, or more precisely Markov Chain Monte Carlo simulation techniques will be used in a crucial way here. Perhaps the best book to learn the mathematical theory from is \cite{AsmussenGlynn07}, or for people who like more text and descriptions \cite{Owen13}. As the ideas here are concerned with risk analysis and system reliability, a third canonical reference is \cite{Zio14}. \cite{AuBeck01} was a key inspiration for the main idea presented here (see also \cite{AuWang2014} and \cite[Section 6.7]{Zio14}). For the convenience of the reader we collect a few basic comments on Monte Carlo simulation theory in appendix \ref{MCappendix}. 

\section{The statistical models}
\label{latency_budget_model}

As explained above we consider a setup where there is a strict requirement that $c$ can produce a good/reliable control within the \emph{latency budget} $T>0$. I.e. we think of time as being discretized into \emph{latency intervals} of length $T$, 
\[
[0, T], (T, 2T], (2T,3T],\ldots,
\]
and $f$ is applied at the right end of each latency interval. We will focus on the analysis of the reliability within a single \emph{random} latency interval, meaning the system state $x\in \X$ is chosen at random. When aiming to carry out the statistical analysis we are faced with two major challenges
\begin{enumerate}
    \item The generation of system states $x\in \X$ in a random latency interval.
    \item The reliable estimation of \emph{rare events} corresponding to solver failures.
\end{enumerate}

\subsection{Generating system states}
\label{section_generating_system_states}
In a first step, we distinguish between embedded control systems operating at a steady state or carrying out transient tasks.
\begin{description}
\item[Steady-State] In many applications it is reasonable to assume that the embedded control system operates around a given \emph{steady-state}, the distribution of which can be found through e.g.\ high-fidelity simulation. The analysis is then reduced to a single distribution of states. 
\item[Transient] In other cases, the task has predominantly transient behavior. In this case distributions for each latency interval can be generated e.g.\ from simulations starting at a distribution of representative initial conditions. The resulting analysis is then carried out for each time interval separately.
\end{description}
Using such high-fidelity simulation, or even real-world experiments, to generate representative state distributions, however, has two significant drawbacks. First, the simulations are typically expensive, such that a generation of a system state sample is several orders of magnitude more expensive than the evaluation of $c$ itself. Ideally we would therefore like to sample cheaply, focusing computational resources on the $c$ itself. Second, in order to apply methodologies addressing the second challenge, i.e. rare event simulation, we typically need an explicit description of the state distribution via a \emph{density}. This then suggests an intermediate step fitting a suitable density to the empirical distribution, for instance via kernel density estimation.

\begin{rem}
\label{synthetic_state_distribution}
To circumvent these issues, one can instead work with fully synthetic state distributions, e.g. a uniform distribution over the operating range of the embedded control system. 
\end{rem}

In the following we will focus our attention on the second challenge (i.e. rare event simulation) and assume we have fixed a distribution on $\X$, represented by a density, from which to sample the system state associated with a random latency interval. The precise setup we consider will be discussed in details below in section \ref{distribution_of_x0}. We will denote the random variable on $\X$ by $X_1$ and its realization by $x_1$. We will now distinguish between two control system designs which will be analysed below.

\subsection{The latency budget model}
If the time it takes to evaluate $c$ is less than $\frac{T}{2}$ one can evaluate $c$ multiple times within each latency interval. The strict requirement for $c$ to be acceptable is that within each interval the dynamical system will find itself in a state $x\in \X$ for which $\F(x)=1$. We will denote by $N$ the number of times we can evaluate $c$ in $T$ time units. E.g. if $T=105\ms$ and it takes $10\ms$ to evaluate $c$ then $N=10$. 
In other words, the requirement we put on $c$ is that, given a sequence of states 
\begin{equation}
\label{state_sequence}
x_1, x_2, \ldots , x_{N},
\end{equation}
encountered within a given latency interval, at least one of the values $\F(x_1), \F(x_2), \ldots, \F(x_N)$ must equal 1.

We will be concerned with computing the probability of experiencing a latency interval within which $c$ fails to meet the requirement, i.e. that a sequence (\ref{state_sequence}) is encountered for which 
\begin{equation}
\label{failure_event}
\F(x_1) = \F(x_2) = \cdots = \F(x_N) = 0.
\end{equation}
We will model the evolution of the dynamics within a single latency interval as a Markov chain. The dependency between the $x_k$s will be modelled by random \emph{independent} perturbations $(\delta_k)_{k=2}^{N}$, in the sense that 
\begin{equation}
\label{markovChain}
X_k = X_{k-1} + \delta_{k}, \quad k = 2,\ldots, N.
\end{equation}
It is natural to also assume $\delta_k$ to be independent from $X_l$ for $l < k-1$, but from a modelling perspective it is convenient to allow $X_{k-1}$ and $\delta_k$ to be dependent. We will denote by $A_k$ the event that $x_k$ gives rise to a bad/unreliable control, i.e.
\begin{equation}
    \label{defAset}
    A_k := ( \F(X_k) = 0 ), \quad k = 1,\ldots, N.
\end{equation}
Below we refer to this statistical model for assessing the reliability of $c$ as the \emph{latency budget model}.

\begin{rem}
It might at first seem strange to model a (at least in theory) deterministic system (\ref{dynamical_system}) by a random one (\ref{markovChain}). The philosophy here goes as follows: The dynamics modelled by (\ref{markovChain}) happen within a single latency interval. As we are interested in real-time control, the evolution within a single latency interval is so small that it seems reasonable to model it as a pure perturbation. This practice of replacing "insignificant" dynamics by random perturbations is quite common in stochastic simulation and there are good reasons for doing it \cite{Schueller07}.
\end{rem}

\subsection{The concurrent design model}
\label{parallel_interpretation}
Today, embedded platforms often offer multiple CPU cores, so it is natural to ask if this can be utilized to robustify $c$. We can do this by using $c$ to design a new "composite" control system running on $N\geq 2$ threads concurrently. Thread number $k$ evaluates $c$ on $y_k$, where
\[
y_k = 
\begin{cases}
x , & \text{if} \ k = 1 \\
x + \delta_{k}, & \text{if}\ k>1
\end{cases}
\]
and $\delta_{k}$ is a perturbation. The $\delta_k$s are generated such that they are pairwise independent and the final control input produced is $c(y_k)$ where $k$ is the smallest thread number for which $\mathbf{F}(y_k)=1$. If such a $k$ doesn't exist, then the system outputs $c(y_N)$. Schematically, we can think of our new control system (in case $N=3$) as shown in Figure \ref{parallel_schematic}.
\begin{figure}[H]
\centering
\[
\vcenter{\hbox{\xymatrix{                   & c(y_1) \ar[rd]          & \\
           x \ar[ru] \ar[rd] \ar[r] & c(y_2) \ar[r]  & u  \\
                             & c(y_3) \ar[ru] &  }}}
=
\begin{cases}
c(y_1),  & \text{if} \ \mathbf{F}(y_1)=1\\
c(y_2), & \text{if} \ \mathbf{F}(y_2)=1 \ \& \ \mathbf{F}(y_1)=0\\
c(y_3), & \text{otherwise}
\end{cases}
\]
\caption{Schematic overview of the concurrent control system.}
\label{parallel_schematic}
\end{figure}
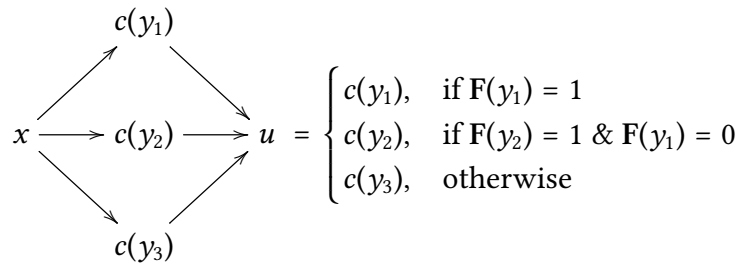
What's important for this control system to make sense is that $c(x + \delta_k)$ is a \emph{good} control at the state $x$ if $\mathbf{F}(x+\delta_k)=1$. We can think of two ways of ensuring this:
\begin{enumerate}[label=(\alph*)]
    \item
    The state $x$ will in many real-world applications be an estimate of the true state of our dynamical system, perhaps obtained using a Kalman filter. Hence, if all the $\delta_k$s are sufficiently small it is reasonable to use $x+\delta_k$ as an estimate for the current true state.
    \item
    In many cases one can think of $c(x)=c(x,x)$ as depending on $x$ in two components. The first component prescribes the constraints of the control problem and the second doesn't. In particular, in such a situation, for any function $h:\mathbf{X} \to \mathbf{X}$, the control produces by $c(x,h(x))$ is still a \emph{good} control at the point $x$ if $\mathbf{F}(x,h(x))=1$. In this situation one can simply perturb the second component on the different threads in Figure \ref{parallel_schematic}.
\end{enumerate}

\begin{rem}
One obvious example that falls into the case (b) is when $c$ is based on embedded optimization. Here, the optimization solver requires at run-time both an input specifying the constraints of the optimization problem (the first input) as well as an initial guess for the algorithm (the second input). In particularly this shows that any system based on embedded optimization fits into the case (b).
\end{rem}

\begin{rem}
Note that the composite system outlined in Figure \ref{parallel_schematic} is stochastic as opposed to the original deterministic $c$. In particular, it might return different outputs given identical inputs. However, the new system is at least as reliable as $c$.
\end{rem}

This trick allows for parallelizing to any number of threads because of the negligible overhead caused by performing the perturbation. The real limiting factor here will be the increased cost of hardware. Given our initial $X_1$ we model $x_k, \ k > 1$ by random variables given by
\begin{equation}
\label{modifiedMarkovChain}
X_k=X_1+\delta_{k}, \quad k=2,\ldots, N. 
\end{equation}
The main difference between this case and the latency budget model outlined above is that $X_k$ ($k>1$) is a perturbation of $X_1$ rather than of $X_{k-1}$. 
Below we will refer to the statistical model given by (\ref{modifiedMarkovChain}) as the \emph{concurrent design model}.

\section{A decomposition technique and subset simulation}
As the aim if this paper is to provide practical methods for verifying reliability of real-time control systems, one expects the probability of the event (\ref{failure_event}) to be extremely small. Hence, we are dealing with estimating the probability of a \emph{rare event}.
A fundamental challenge in Monte Carlo simulation is that (in the simplest setting), to verify $\P(Z \in A) \approx p$ for a random variable $Z$, $p\in (0,1)$ and $A\subset \X$ one needs to simulate $\approx \frac{1}{p}$ independent samples of $Z$ \cite[Chapter III]{AsmussenGlynn07} (see also appendix \ref{MCappendix}). Hence, in the case of a rare event $(0 < p \ll 1)$ the required compute becomes prohibitively large. A typical strategy when dealing with rare event simulation is the use of variance reduction techniques \cite[Chapter V]{AsmussenGlynn07}, \cite[Chapter 8]{Owen13}. These often come at the price of biasing estimators. In this section we discuss how to estimate the probability of the failure event (\ref{failure_event}) using the \emph{subset simulation} technique \cite{AuBeck01, AuWang2014}. This provides an estimator for the probability of (\ref{failure_event}) with small bias and good convergence properties. This renders the probability of (\ref{failure_event}) computationally feasible even when it is extremely small (e.g. $\approx 10^{-15}$).

\begin{prop}
\label{prop1}
In both model (\ref{markovChain}) and (\ref{modifiedMarkovChain}) we have
\begin{equation}
\P (\cap_k A_k ) = \P(A_1) \prod_{k=2}^{N} \P(A_k | \cap_{l<k} A_l)
\end{equation}
\end{prop}
\begin{proof}
Simply use the definition $\P(\cdot \ |\  \cdot )$ inductively.
\end{proof}
This proposition allows us to "split" the failure event (\ref{failure_event}) into $N$ sub-events 
\begin{equation}
\label{sub_events}
(A_1) \quad \text{and} \quad (A_k | \cap_{l<k}A_l), \quad k=2,\ldots, N.
\end{equation}
We will estimate the probability of each of these events separately, denoting by $Z_1$ an estimator for $\P( A_1 )$ and by $Z_k$ an estimator for $\P ( A_k | \cap_{l<k}A_l )$ for each $k=2,\ldots, N$. 
\begin{prop}
\label{prop2}
Assume that
\begin{enumerate}[label=(\alph*)]
    \item 
    All the $Z_k$s are centered, i.e.
    \[
    \E(Z_1) = \P(A_1)\ \text{and} \ \E(Z_k) = \P( A_k | \cap_{l<k}A_l ), \ k=2,\ldots, N.
    \]
    \item
    All the $Z_k$s are independent.
\end{enumerate}
Then 
\begin{equation}
    \label{estimator}
    Z:=Z_1 \cdots Z_{N}
\end{equation}
is a centered estimator for $\P( \cap_k A_k )$. Moreover, if $K > 0$ such that 
\begin{equation}
\label{variance_condition}
(2^N - 2)\E(Z_k)^2\leq K\V(Z_k) \quad \forall \ k=0,\ldots, N-1 
\end{equation}
then 
\begin{equation}
\label{variance_estimate}
\V(Z)\leq (1+K^{N-1})\V(Z_1)\cdots \V(Z_N).
\end{equation}
\end{prop}
\begin{proof}
Use Proposition \ref{prop1} to see that $Z$ is centered. The variance estimate comes from
\begin{align}
    \V(Z)&= \E(Z_1^2)\cdots \E(\Z_N^2) - \E(Z_1)^2\cdots \E(\Z_N)^2 \label{var_computation} \\
    &=\prod_{k=1}^{N}(\V(Z_k) + \E(Z_k)^2) - \prod_{k=1}^{N}\E(Z_k)^2 \nonumber \\
    &= \V(Z_1)\cdots \V(Z_{N}) + \cdots \nonumber \\
    &\leq \V(Z_1)\cdots \V(Z_{N}) + K^{N-1}\V(Z_1)\cdots \V(Z_{N}) \nonumber \\
    &= (1+K^{N-1})\V(Z_1)\cdots \V(Z_{N}). \nonumber
\end{align}
\end{proof}
The point of Proposition \ref{prop2} is that we can build an unbiased estimator $Z$ for $\mathbb{P}(\cap_{k}A_k)$ from the unbiased estimators $(Z_k)_k$. Moreover, (\ref{variance_estimate}) provides a way to estimate the variance of $Z$ in an unbiased manner using unbiased and independent estimators for the $\mathbb{V}(Z_k)$s. This can be used to compute a rigorous confidence interval for $\mathbb{P}(\cap_{k}A_k)$ using a concentration inequality \cite{Boucheron13}. Typically however, it will be more practical to simply compute confidence intervals based on each $Z_k$ using the central limit theorem and combine these intervals into a confidence interval for $Z$. In fact in all examples below this is how we will compute confidence intervals for $\mathbb{P}(\cap_{k}A_k)$.

Let's now explain how the above setting fits into the subset simulation framework presented in \cite{AuBeck01}. For this define 
\[
F_k := \bigcap_{l\leq k} A_l, \quad k=1,\ldots ,N,
\]
so that $F_1 \supset \cdots \supset F_N$. Then the sub-events (\ref{sub_events}) can be written
\[
F_1 = A_1 \quad \text{and} \quad (F_k\ |\ F_{k-1}) = ( A_k\ |\ \cap_{l<k}A_l ), \quad k=2,\ldots, N.
\]
This is exactly the situation covered in \cite{AuBeck01}, where estimators $Z_k$ for the $\P(F_k\ |\ F_{k-1})$s are constructed. Each of these estimators are unbiased, but they are pairwise dependent because realizations of $Z_{k-1}$ are used as initial conditions for the Markov chain simulation of $Z_k$. This in turns causes the product $Z = Z_1\cdots Z_N$ to be biased. However, as explained in the following result, this bias is negligible and $Z$ is asymptotically unbiased with good convergence properties.

\begin{prop}[ {\cite[Proposition 1 and 2]{AuBeck01}} ]
\label{prop_au_beck}
     Let $Z_k$ be the estimator for $\P(F_1)$ if $k=1$ and for $\P(F_k\ |\ F_{k-1})$ if $k>1$, which are constructed in \cite[Section 4]{AuBeck01}. Then, with $Z:=Z_1 \cdots Z_N$ we have
     \begin{align*}
        | \E\left( \frac{Z - P}{P} \right) | &= O(M) \\
        \E\left( \frac{Z - P}{P} \right)^2 &= O(M),
     \end{align*}
     where $P:=\P(\cap_{k}A_k)=\P(F_N)$ is the failure probability and $M$ is the total number of samples.
\end{prop}
The main point in this result is that the estimator $Z$ for $\P(F_N)$ proposed in \cite{AuBeck01} both has good convergence results and a negligible bias. For further details on the constructions of the $Z_k$s we refer to section \ref{sec_defining_the_estimators} and remark \ref{inital_conditions} below.

\section{Making the models computational}
\label{model_section}
Let's discuss how to implement a statistical framework based on our statistical models. To do this we will, throughout the rest of this paper, assume either 
\begin{enumerate}[label=\textbf{A\arabic*}]
\item \label{assumption1}
$\X = \X_I \times \X_M \times \X_D$, where $\X_I$ is a product of closed intervals, $\X_M$ is a boundary-less submanifold of euclidean space from which we can sample and perturb uniformly and $\X_D$ is a discrete subset from which we can sample uniformly, or
\item \label{assumption2}
$\X\subset (\X_I \times \X_M \times \X_D)^m$ is a subset of the $m$-fold product of such a space with the condition that 
\[
x = (x^1,\ldots , x^m)\in \X
\]
exactly when 
\begin{equation}
    \label{AssumtionBAdditional}
    |x^i_J-x^j_J| > r_{ij} \geq 0 \quad \forall  \ i\neq j,
\end{equation} 
here $J$ denotes a subset of the indices of the components of $x\in \X$ and $|\cdot |$ denotes a distance on those variables.
\end{enumerate}

\begin{rem}
    The case of \ref{assumption2} is a natural generalization of \ref{assumption1} to the case when there are multiple actors/obstacles within the same world. E.g. if we are trying to control a quadrotor (see subsection \ref{quadcopterExample}) flying inside a box $[-50, 50]^3\subset \R^3$, then \ref{assumption2} corresponds to having multiple quadrotors inside $\X$ and (\ref{AssumtionBAdditional}) corresponds to the natural requirement of collision avoidance (by choosing $J$ to correspond to the position variables).
\end{rem}

Assumptions \ref{assumption1} and \ref{assumption2} allow us to sample uniformly from $\X$ as well as perturb within $\X$. We will denote by $U$ the uniform distribution on $\X$. In the case of \ref{assumption2} sampling from $U$ is done by applying the Metropolis algorithm, see \cite{Metropolis1953}.
We will now explain how to perturb:
Under \ref{assumption1}, given $x=(x_I, x_M, x_D)\in \X_I \times \X_M \times \X_D$ and $r = (r_I, r_M, r_D)\in [0,\infty)^3$ we will denote by $\mathcal{D}(x, r)$ the distribution given as follows: On the $\X_M$ and $\X_D$ components it is just a uniform perturbation of radius $r_M$ and $r_D$ respectively. On the $\X_I$ component it is also uniform perturbation of radius $r > 0$ with the condition that the trajectory from $x_I$ is reflected at the boundaries of $\X_I$ like a billiard ball. In case of \ref{assumption2} we will abuse notation and also write $\mathcal{D}(x,r)$ for a product distribution, each of the $m$ components being of the type just described.

\subsubsection{Sampling the $X_k$s}
\label{distribution_of_x0}
Once we can sample $X_{k-1}$ for some $k\geq 2$, sampling $X_{k}$ comes down to sampling $\delta_{k}$. We do this by simulating $\delta_k \sim \mathcal{D}(X_{j}, r)$ for some $r$, with $j=k-1$ in the latency budget model and $j=1$ in the concurrent design model. Note that within the concurrent design model this sample follows the "true" distribution if this is also how the perturbation within the composite system is implemented. In the case of the latency budget model this choice of perturbation is justified if the dynamics within a single latency interval are very insignificant (because $T$ is small compared to the speed of the real-world dynamics).

As mentioned in sections \ref{latency_budget_model} and \ref{section_generating_system_states} one aspect of our analysis is simulating the random variable $X_1$, representing the system state within a random latency interval. The synthetic and easiest solution is to simply choose $X_1 \sim U$. However, in order to yield a direct interpretation in terms of \emph{expected time until failure} one must sample from the "real-world distribution" of $X_1$. The most general situation under which this can be done within the current framework is when this distribution is described by a density $f_{X_1}$ with respect to $U$ on $\X$. In this case, under \ref{assumption1}, we will sample $X_1$ using the \emph{independence sampler} \cite[Chapter XIII, Section 4a]{AsmussenGlynn07}. This is a variant of the Metropolis-Hastings algorithm which produces a sequence $(Y_j)_{j=1}^K$ whose asymptotic distribution is that of $X_1$. The update rule is particularly simple in this case. Given $Y_j$ the update evolution is prescribed by the following rule: Sample $W \sim U$ and set 
\begin{equation}
\label{independece_sampler}
Y_{j+1}=
\begin{cases}
W, & \text{with probability}\ p:=\min\left( 1, \frac{f_{X_1}(W)}{f_{X_1}(Y_j)}\right) \\
Y_j, & \text{with probability}\ 1-p
\end{cases}
\end{equation}
One can also do a similar sampling technique under \ref{assumption2} by sampling $X_1$ from the symmetric random walk Metropolis algorithm \cite[Chapter XIII, Section 4.b]{AsmussenGlynn07}. However, we will not discuss this further here. In fact in the examples below we will exclusively sample from the uniform distribution on $\X$.

\subsection{Defining the estimators}
\label{sec_defining_the_estimators}
Given that we can now simulate the $X_k$s, let's discuss how to define the estimators $(Z_k)_{k=1}^{N}$ in Propositions \ref{prop2} and \ref{prop_au_beck}.
$Z_1$ will be a standard Monte Carlo estimator. Since we can sample from the distribution of $X_1$ (either $U$ or $f_{X_1}\cdot U$, see section \ref{distribution_of_x0}), the obvious candidate is
\begin{equation}
\label{z0_estimator}
Z_1=\frac{1}{K}\sum_{j=1}^{K} \1 \{ \F(Y_j) = 0 \}.
\end{equation}
Here $(Y_j)_{j\in \N}$ is a sequence of random variables, either\footnote{In the terminology of \cite{AsmussenGlynn07} this is the \emph{crude} Monte Carlo estimator for $\P( \F(X_1) =0 )$.} I.I.D. with $Y_j \sim U$, or whose asymptotic distribution is that of $X_1$.

The estimators $Z_k,\ k\geq 2$ will be built using a \emph{batched} version of the classical random walk Metropolis algorithm (RWM) \cite{Metropolis1953} \cite{Mackay02} \cite[Chapter XIII, Section 3]{AsmussenGlynn07}. To see how this comes about, denote by $P_k$ the joint distribution of $(X_1,\ldots, X_k)$ for $k=1,\ldots, N$. We can then write
\begin{align}
\P(A_k | \cap_{l<k}A_l) &= \frac{1}{\P ( \cap_{l<k}A_l )} \int_{\{ \cap_{l<k}A_l \}} \1 \{ \F(X_{k-1}+\Delta Y)=0 \} \ d\P \label{cond_prob}\\
&= \int \1 \{ \F(Y_{k-1}+\Delta Y)=0 \} \ d\P, \nonumber
\end{align}
where\footnote{Given a probability measure $Q$ and a measurable set $A\subset \X$ with $Q(A)>0$, we denote by $Q|_A$ the probability measure given by $Q|_A(B)=\frac{Q(A \cap B)}{Q(A)}$.} $Y=(Y_1,\ldots, Y_{k-1}) \sim P_{k-1}|_{\{ \cap_{l<k}A_l \}}$ and $\Delta Y \sim \mathcal{D}(Y_{k-1},r)$. Hence, the obvious estimator for $\P(A_k | \cap_{l<k}A_l)$ is 
\begin{equation}
\label{zk_estimator}
\frac{1}{K}\sum_{j=1}^K \1 \{ \F( Y_{j,k-1} + \Delta Y_j )=0 \}
\end{equation}
where $(Y_j, \Delta Y_j)_{j=1}^{K}$ is a sequence of random variables with
\begin{equation}
\label{sampling_sequence}
Y_j \sim P_{k-1}|_{\{ \cap_{l<k}A_l \}} \quad \text{and} \quad \Delta Y_j\sim \mathcal{D}(Y_{j},r) \quad \forall \ j\in \N
\end{equation}
for some $r$. The challenging bit here is sampling from $P_{k-1}|_{\{ \cap_{l<k}A_l \}}$.

\begin{ex}
    \label{simple_sampling}
    There is one simple (albeit inefficient) way to produce the $Y_j$s: Sample independent samples of $(X_1,\ldots, X_{k-1})$ (perhaps using (\ref{independece_sampler})) and record the samples for which $\F(X_1)=\cdots = \F(X_{k-1})=0$. The sequence thus recorded is a collection of independent samples from the distribution $P_{k-1}|_{\{ \cap_k A_k \}}$. Although highly inefficient, this "brute-force" approach provides a way for computing convergence diagnostics (see section \ref{conv_diag_sec} below).
\end{ex}

We will now discuss how to sample $P_{k-1}|_{\{ \cap_{l<k}A_l \}}$ in an efficient way using RWM. More precisely we use RWM to produce a trajectory $(y_j)_{j=1}^K$ of a Markov chain whose asymptotic distribution is that of $P_{k-1}|_{\{ \cap_{l<k}A_l \}}$ (i.e. this will be the realization of the $Y_j$s above). Recall that in RWM the perturbations $t$ of the Markov chain produced by RWM are realizations of a \emph{proposal distribution}, see e.g. Appendix \ref{MCappendix}. The perturbation is then subsequently accepted/rejected with a given probability. 

In short, given the current iterate $y_j=(y_{j,1},\ldots, y_{j,k-1})$ we sample $t \in \mathcal{D}(y_{j,1}, r_{RWM})$ (for some chosen $r_{RWM}$) which defines the first coordinate of the temporary iterate $y^{tmp}_{j+1}$, $y^{tmp}_{j+1, 1}:=y_{j}+t$. The subsequent coordinates of $y^{tmp}_{j+1,l}, \ l \geq 2$ are generated either by the evolution (\ref{markovChain}) or (\ref{modifiedMarkovChain}), depending on which model we are using. The iterate is then updated according to the rule\footnote{N.B. in the case where $X_1\sim U$ this sets $y_{j+1}=y^{tmp}_{j+1}$ if and only if $\F(y^{tmp}_{j+1,1})=\cdots=\F(y^{tmp}_{j+1,k-1})=0$.}
\begin{equation}
\label{rwm_update}
y_{j+1} = 
\begin{cases}
y^{tmp}_{j+1}, & \text{with probability}\ p:=\frac{\1\{ \F(y^{tmp}_{j+1,1})=\cdots=\F(y^{tmp}_{j+1,k-1})=0 \}f_{X_1}(y^{tmp}_{j+1,1})}{f_{X_1}(y_{j,1})} \\
y_j, & \text{with probability}\ 1-p.
\end{cases}
\end{equation}
The choice of $r_{RWM}$ is a design choice. Note that sampling $(Y_j)_{j\in \N}$ in (\ref{sampling_sequence}) as a Markov chain means that also 
\[
( Y_{j,k-1} + \Delta Y_j)_{j\in \N} 
\]
is a Markov chain, so the Markov Chain version of the Central Limit Theorem applies and we can compute approximate confidence intervals for (\ref{cond_prob}) based on (\ref{zk_estimator}), using a normal distribution, if we have a suitable estimator for the variance term.
For the variance estimate we use the method of batch means, which simply means that the estimator $Z_k$ for $\P(A_k | \cap_{j<k}A_j)$, that we will use, is given by 
\begin{equation}
\label{zk_estimator_batch}
Z_k = \frac{1}{M} \sum_{i=1}^M \overline{Z}^i_{k},
\end{equation}
where each $\overline{Z}^i_{k}$ is an estimator of the type (\ref{zk_estimator}). 

\begin{rem}
\label{inital_conditions}
When the initial conditions $(y_1^i)_{i=1}^M$ for the simulations of the $\overline{Z}^i_{k}$s are obtained by "honest", independent samples from $P_{k-1}|_{ \{ \cap_{j<k}A_j \} }$, obtained via the strategy in Example \ref{simple_sampling}, then all the $Z_k$s are pairwise independent and we are in the case of Proposition \ref{prop2}. If instead the samples from $Z_{k-1}$ which happen to land in $\cap_{l<k}A_l$ are used as initial conditions for the $\overline{Z}^i_{k}$s, then the $Z_k$s are no longer pairwise independent. However, this is exactly the case covered by Proposition \ref{prop_au_beck}, which guarantees good statistical properties.
\end{rem}

When estimating $\P(A_k | \cap_{j<k}A_j)$ using (\ref{zk_estimator_batch}) we use the following batch variance estimate
\begin{equation}
    \label{batch_variance_estimate}
    V_k = \frac{1}{M-1} \sum_{i=1}^M (\overline{Z}^i_k - Z_k)^2.
\end{equation}
With the estimators $Z_k$ and $V_k$ at hand an approximate $\alpha$-confidence interval for $\P(A_k | \cap_{j<k}A_j)$ is given by 
\begin{equation}
\label{student_t_confidence_interval}
z_k \pm t^{\alpha + (\tfrac{1-\alpha}{2})}_{M-1}\sqrt{\tfrac{v_k}{M}},
\end{equation}
where $z_k$ is the realization of (\ref{zk_estimator_batch}), $v_k$ is the realization of (\ref{batch_variance_estimate}) and $t^{\beta}_{N}$ is the $\beta$ percentile of the student $t$ distribution with $N$ degrees of freedom. 
See e.g. \cite[Section 11.12]{Owen13} and \cite[Section 1.10.1]{MCMCIntrobrooksGeyer} for nice discussions on this approach to computing approximate confidence intervals.

\begin{rem}
\label{batching_remark}
Note that this batching approach means that the number of samples on which our final estimator will be based on is $KM$. If we require a total of $L \gg 0$ samples to have high confidence in our statistical estimator then we must choose $K=\frac{L}{M}$. Hence, if $M=L$ then the Metropolis samples above consist of a single sample (i.e. the initial value) and our estimator $Z_k$ is an honest estimator (i.e. follows $P|_{\{ \F = 0 \}}$ exactly), at least when $P=U$ (i.e. $f_{X_0}=1$). However, taking $M=L$ will often not be possible because of the compute required (see Example \ref{simple_sampling}). However, it shows that scaling up the compute power can significantly improve the reliability of the estimators.
\end{rem}
For the convenience of the reader we summarize the algorithm for computing the value of the $\overline{Z}_k^i$s. I denote by $\mathcal{U}$ the uniform distribution on $(0,1)$.

\begin{algorithm2e}
\SetAlgoLined
\KwResult{ A sequence $(z_k)_{k=1}^K$ of realizations of $\mathds{1}\{ \mathbf{F}(Y + \Delta Y) = 0 \}$ from (\ref{zk_estimator}). }
\KwData{$y_1\in \mathbf{X}$, a sample from $P|_{\{ \mathbf{F} = 0 \}}$. $r_p$ prescribing the size of the perturbations in (\ref{markovChain}) and $r_{RWM}$ prescribing size of the perturbation of the RMW Markov chain.}
$y \leftarrow y_1$\;
\For{$1 \leq k \leq K$}{
$z_{k} \leftarrow 0$\;
Sample $\Delta y\sim \mathcal{P}(y,r_p)$ \;
\If{$\mathbf{F}( y + \Delta y ) = 0$}{
$z_{k} \leftarrow 1$ \;
}
Sample $t \sim \mathcal{D}(y,r_{RMW})$ \;
Sample $u \sim \mathcal{U}$\;
\If{$u < \min\left( 1, \frac{\mathds{1} \{ \mathbf{F}(y+t)=0 \}f_{X_1}(y+t)}{f_{X_1}(y)} \right)$}{
$y \leftarrow y + t$ \;
}
}
\caption{$\mathds{1} \{\mathbf{F} ( Y + \Delta Y) = 0 \}$ - sampling}
\label{ModifiedMetropolisAlgorithm}
\end{algorithm2e}

\begin{rem}
    Note that if $f_{X_1} \equiv 1$, so we are sampling from the uniform distribution on $\X$, then the last update occurs if and only if $\F( y_{k} + t_k ) = 0$. In fact this is the only case we will consider in the examples below. 
\end{rem}

\subsection{Convergence diagnostics}
\label{conv_diag_sec}
How do we know when the RWM has converged? This is a fundamental question in Markov Chain Monte Carlo methods and in general there is no perfect answer. Note that while $r_p$ is a property of the models (\ref{markovChain}) and (\ref{modifiedMarkovChain}) in Algorithm \ref{ModifiedMetropolisAlgorithm}, the relevant quantity concerning convergence properties of the algorithm is $r_{RWM}$. Too small values in the components of $r_{RWM}$ will cause the development of the Markov chain to halt and its will not converge to its stationary distribution. Too large values in the components of $r_{RWM}$ will cause very few updates of $y$ to happen which will also halt the Markov chain. Within the setting of this paper there is in general no criteria to ensure convergence has happened. That being said, there is the inefficient way to ensure that $Y$ is an honest sample from $P|_{\{ \F = 0 \}}$ (see Remark \ref{batching_remark}). Hence, a good way to do convergence diagnostics is to compare the statistics computed on the basis of honest samples with statistics computed from the RWM samples. The most straightforward approach (and the approach we advocate here) is to choose so-called \emph{pilot functions} $p_l:\X \to \R$, $l=1,\ldots,L$ and compare the means 
\[
\overline{p}_l(x) = \frac{1}{H}\sum_{k=1}^H p_l(x_k) \quad \text{and} \quad \overline{p}_l(y) = \frac{1}{K}\sum_{k=1}^K p_l(y_k),
\]
obtained form an honest sample $x=(x_k)_{k=1}^H$ and the RWM sample $y=(y_k)_{k=1}^K$, coming from Algorithm \ref{ModifiedMetropolisAlgorithm}. One could even compute confidence intervals based on $\overline{p}_l(x)$ using the central limit theorem and require to find $\overline{p}_l(y)$ within that confidence interval. See also \cite[Chapter XII, section 3b]{AsmussenGlynn07} and references therein for further discussions on this topic. In applications to statistical mechanics one can sometimes choose $p_l$s whose true mean $\lim_{H\to \infty}\overline{p}_l(x)$ one can predict from theoretical considerations. This will in general not the case here, so we suggest to simply choose the $p_l$s at random, e.g. choose them as polynomials of a given degree, whose coefficients are selected at random. For more ideas on how to do convergence diagnostics for Monte Carlo applications to reliability analysis, see \cite{AuBeck01, Zio14}.

\begin{rem}
    The Metropolis algorithm is one of the great scientific achievements of the last century \cite{Cipra00} and there are plenty of resources discussing its convergence properties, see in particular \cite{AsmussenGlynn07}, \cite{Owen13} and \cite{MCMCIntrobrooksGeyer} as well as the references therein for very readable accounts.
\end{rem}

\subsection{Interpreting $Z$}
\label{simulation_time_considerations}
The above subsections outline a statistical framework for estimating the mean $\E(Z)$ (see (\ref{estimator})). The framework is designed in such a way that it is possible to verify that 
\[
    \P( \cap_k A_k ) = \E(Z)
\]
is extremely small with realistic computational resources. How do we interpret such a result? One would like to conclude that the \emph{expected time until failure} is extremely large. This interpretation is correct if the distribution of $X_1$ is the \emph{true} distribution of the state of our system. See also the discussion in section \ref{section_generating_system_states}.

To illustrate how this interpretation comes about in the latency budget model, assume that $X_1$ does carry the true distribution of the state of our system, that $\F$ takes $10\ms$ to evaluate on a given single core machine and that $T=25\ms$ (so $N=2$). Then choosing $K=10^7$ it will take roughly $1.17$ days to compute our estimator $Z_1$. Each iteration in Algorithm \ref{ModifiedMetropolisAlgorithm} requires two evaluations of $\F$, so taking the same $K$ here, it will take roughly an additional $2\cdot 1.17=2.34$ days to compute the estimator $Z_1$. These $\approx 3.5$ days of compute allows one, in principle, to verify $\E(Z_1) \approx 10^{-6}$ and $\E(Z_1) \approx 10^{-6}$, which in turn gives an estimate for the failure probability bounded above by $\approx 10^{-12}$.
The interpretation of such a result is that, on average, $1$ in $10^{12}$ latency intervals will experience failure, or in other words that the expected time between two consecutive failures is
\[
    10^{12}\cdot 25\ms > 792 \years.
\]
Note that this is a computation which is realistic to perform on a single core machine!

To illustrate how the "expected time until failure" interpretation comes about in the concurrent design model, suppose we are in the same case, so it takes $10ms$ to evaluate $c$ and let's say $T=10ms$, so $N = 1$. Once again taking $K = 10^{7}$, it will require $\approx k \cdot 1.17$ days to simulate $Z_k$ (because it requires $k$ evaluations in each loop of Algorithm \ref{ModifiedMetropolisAlgorithm}). This allows us in principle to verify 
\[
    \E(Z_k) \approx 10^{-6} \quad k > 0.
\]
Hence, if our "composite control system" runs $c$ on three threads (as in Figure \ref{parallel_schematic}) one can, in principle, verify a failure probability of $\approx 10^{-18}$ with $\approx (1+2+3)\cdot 1.17=7.02$ days of compute. This failure probability translates into the following expected time between two failure events:
\[
    10^{18} \cdot 10\ms > 317\cdot 10^{6} \years.
\]

\section{Examples}
Here we discuss a few examples. The $c$ in each of these examples is a nonlinear MPC controller implemented in \href{https://www.embotech.com/products/forcespro/overview/}{FORCESPRO}. 

\subsection{A toy example}
Consider the following control problem: Given an initial condition $x_{0}\in \X:=[-8,8]^2$ we wish to steer our dynamical system to the origin $(0,0)\in \X$ using a 1-dimensional control input $u$. The dynamics are given by the differential equations
\begin{align*}
\dot{x}_1 &= x_2 \\
\dot{x}_2 &= u( 1 - x_1^2 )x_2 - x_1
\end{align*}
This system is often called the \href{https://en.wikipedia.org/wiki/Van_der_Pol_oscillator}{Van der Pol oscillator}. Our control algorithm is a 10-stage interior point NLP solver formulated in FORCESPRO. As objective function we use 
\[
10^{-5}u^2 + x_1^2 + x_2^2
\]
on each stage and sum up the contribution of the different stages.
Note that even though the objective function is convex, the optimization problem is clearly non-convex. The FORCESPRO solver requires two inputs whenever we call it: The current state of our system (an element of $\X$) as well as an initial guess for the primal variable of our optimizations algorithm (this is a 30 dimensional vector).

\begin{rem}
    The Van der Pol oscillator might seem like a simple system because the state space is 2 dimensional, but the dynamics are known to be unstable and the origin is known to be an unstable equilibrium. Moreover the optimization problem is clearly non-convex, so in that sense the problem contains all the theoretical ingredients required to be a challenging control problem.
\end{rem}

\begin{rem}
    The simulations in this example were performed on a laptop PC with a four (i5-3380M) core Intel CPU, 2.90 GHz, 3-4MB cache, running 64-bit Ubuntu 20.04.1 LTS.
\end{rem}

\subsubsection{Convergence diagnostics}
As this system is two dimensional we can have a chance to investigate what a good value for $r_{RWM}$ is. This has been done by comparing visually the distribution of the restriction to $\{\F=0\}$ of the uniform probability distribution on $\X$ with samples of this probability measure obtained by the RWM. In Figure \ref{vanDerPolDistFig} we visualize the distribution of $\1 \{ \F(X_1) = 0 \}$ via a histogram based on $10^9$ uniformly distributed samples in $\X$, 305 of which landed in $\{ \F = 0 \}$. As can be seen points at which $\{ \F = 0 \}$ are extremely sparse, so the probability of getting two consecutive samples in this set is extremely unlikely. Figure \ref{metropolisVanDerPolDistFig} shows the empirical distribution obtained by running the RWM with $r_{RWM}=1.0$ starting from 31 initial conditions in $\{\F=0\}$. These were obtained by sampling uniformly $10^7$ samples in $\X$ and collecting the samples landing in $\X$. As can be seen from the figures the RWM is closed to having converged. 

\begin{figure}
\label{vanDerPolDistFig}
\centering
\includegraphics[scale=0.6]{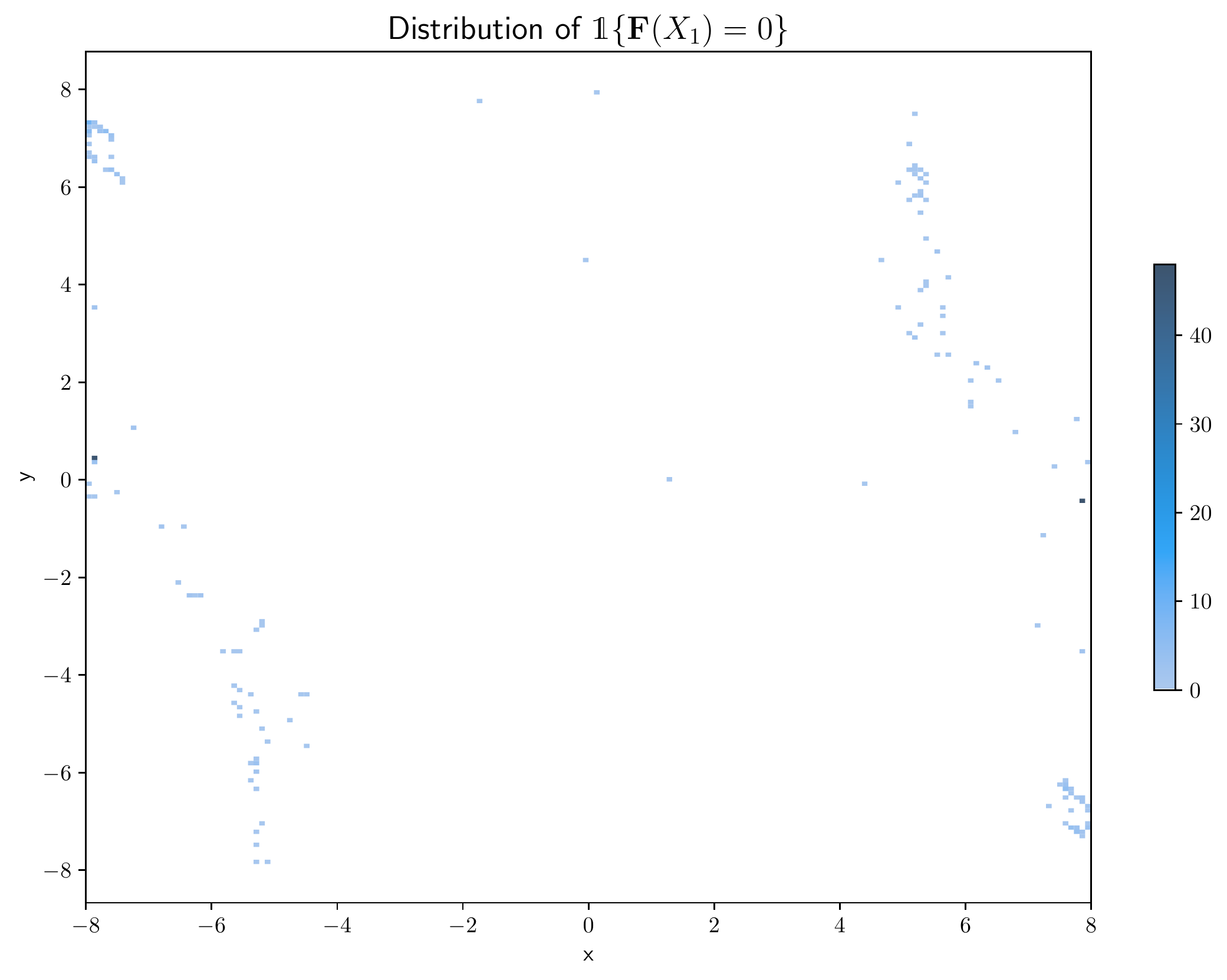}
\caption{A histogram based on $10^9$ independent samples of the uniform distribution in $\X$ in the Van der Pol example. Out of these $10^9$ samples $305$ samples were in $\{\F = 0\}$ and the above histogram is based on these $305$ samples. The intensity of the color indicates the number of samples found in a given square.}
\end{figure}

\begin{figure}
\label{metropolisVanDerPolDistFig}
\centering
\includegraphics[scale=0.6]{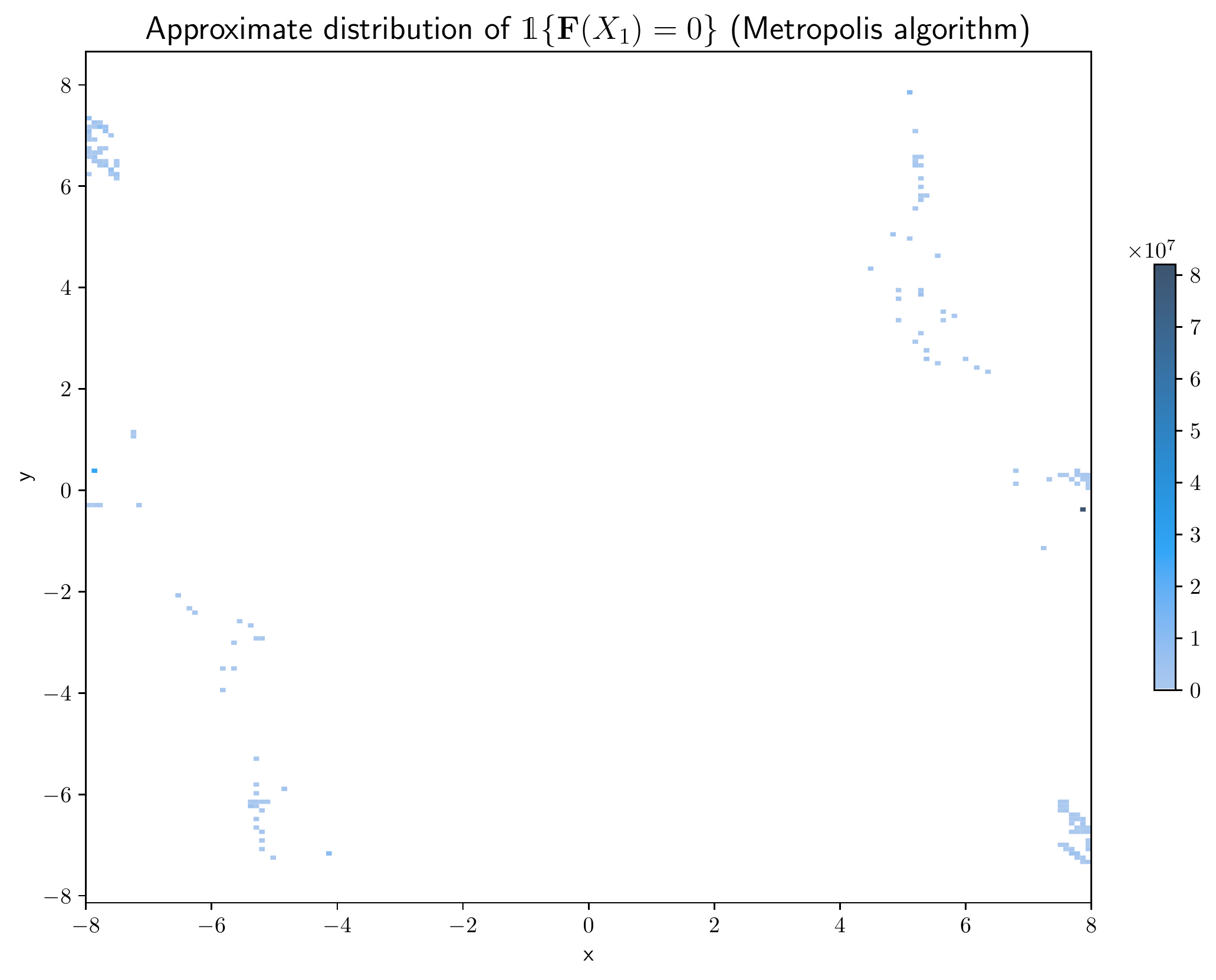}
\caption{A histogram based on $3.1\cdot 10^8$ samples obtained from the Metropolis algorithm as the approximate distribution of $\1 \{ \F(X_1)=0 \}$ in the Van der Pol example. Note the $\times 10^7$ on the scale to the right. Each sample of the Metropolis random walk is run $10^7$ steps and the 31 initial conditions were obtained by independently sampling the uniform distribution on $\X$. The remarkable resemblance between this and Figure \ref{vanDerPolDistFig} indicates the Metropolis algorithm as converged.}
\end{figure}

\subsubsection{Latency budget model results}
Given $x\in \X$ we use the strategy of giving 
\[
(0,x^T,0,x^T, \ldots , 0,x^T)^T
\]
as initial guess to our solver. This renders our FORCESPRO MPC controller as a map $\X \ni x \mapsto c(x)$ and we set $\F(x)=1$ if and only if our FORCESPRO solver terminates with exitflag $1$ and $\F(x)=0$ otherwise. We will assume $N=2$ and $Z_1$ will be the crude Monte Carlo estimator while $Z_2$ will be computed using Algorithm \ref{ModifiedMetropolisAlgorithm} with 
\begin{equation}
\label{van_der_pol_perturbations}
r_p = (0.15, 0.15) \quad \text{and} \quad r_{RWM} = (1.0, 1.0).
\end{equation}

We compute our statistical estimates using confidence level 
\begin{equation}
\label{van_der_pol_alpha}
\alpha := 1-10^{-6}.
\end{equation}
The simulation for $Z_1$ is based on $10^8$ independent samples from $P$ and yielded the following $\alpha$-confidence interval for $\P(A_1)$:
\[
(2.21\cdot 10^{-8}, 5.58\cdot 10^{-7}).
\]
For the estimate for $\P(A_2 \ |\ A_1)$ we have 31 initial conditions and obtain 
\[
z_2 = 5.74 \cdot 10^{-5}, \quad v_2 = 2.65 \cdot 10^{-9},
\]
which (using \ref{student_t_confidence_interval}) yields the following $\alpha$-confidence interval for $\P(A_2 \ |\ A_1)$:
\[
(8.21 \cdot 10^{-7}, 1.14\cdot 10^{-4})
\]
Adding up the results, our analysis provides the $\alpha$-upper bound 
\[
1.14 \cdot 10^{-4} \cdot 5.58\cdot 10^{-7}=6.36 \cdot 10^{-11}
\]
for the probability $\P(A_1\cap A_2)$ of the failure event (\ref{failure_event}). 

\subsubsection{Concurrent design results}
We think of our FORCESPRO interior point solver as a map
\[
\X \times \X \ni (x,y) \mapsto c(x,y),
\]
where $c(x,y)$ denoted the solution found by our FORCESPRO solver using $x$ as the state of our system and initial guess $(0,y^T,0,y^T, \ldots , 0,y^T)^T$. This allows us to build a control algorithm by running $c$ on three threads as depicted in Figure \ref{parallel_schematic}, using the strategy (b) on page \pageref{parallel_schematic}. We will use confidence level $\alpha=1-10^{-6}$, the crude Monte Carlo estimator (\ref{z0_estimator}) for $Z_1$ and set the parameters
\begin{equation}
\label{van_der_pol_perturbations}
r_p = (0.3, 0.3) \quad \text{and} \quad r_{RWM} = (1.0, 1.0).
\end{equation}
in the implementation of Algorithm \ref{ModifiedMetropolisAlgorithm}. Estimation of $\P ( A_1 )$ is identical to the estimation above, so the $\alpha$-confidence interval we achieve is given by
\[
(2.21\cdot 10^{-8}, 5.58\cdot 10^{-7}).
\]
For the estimation of $\P(A_2\ |\ A_1)$ we start with $31$ initial conditions and obtain
\begin{align*}
z_2 = 2.98\cdot 10^{-5}  \quad \text{and} \quad v_2 = 5.18 \cdot 10^{-10}.
\end{align*}
Using (\ref{student_t_confidence_interval}) this yields the following $\alpha$-confidence interval for $\P(A_2\ | \ A_1)$:
\[
(4.80 \cdot 10^{-6}, 5.48 \cdot 10^{-5})
\]
For the estimation of $P(A_3 | \cap_{j=1}^2 A_j)$ we start with $924$ initial conditions and obtain 
\begin{align*}
z_3 = 5.67 \cdot 10^{-5}  \quad \text{and} \quad v_3 = 5.97 \cdot 10^{-9},
\end{align*}
yielding the $\alpha$-confidence interval
\[
(4.42 \cdot 10^{-5}, 6.92 \cdot 10^{-5})
\]
As a consequence we get the following statistical upper bound for the probability of the failure event:
\[
\P(A_1 \cap A_2 \cap A_3) \leq 5.58 \cdot 10^{-7} \cdot 5.48 \cdot 10^{-5} \cdot 6.92 \cdot 10^{-5} \leq 2.12^{-15}
\]

\subsection{A quadrotor example}
\label{quadcopterExample}
Consider now the control problem of stabilizing a quadrotor at its hover state. Again we use FORCESPRO to design our real-time controller. The original design of how to implement this using FORCESPRO is published in \cite{Domahidi17}. The original model for the dynamics is discussed in more details in \cite{Quirynen13} and \cite{Hoffman07}. The statespace in our model consists of $p\in \R^3$ and $v\in \R^3$, denoting the position and velocity respectively, $q\in \S^3$ and $\Omega \in \R^3$ denoting the orientation expressed in quaternion representation and its angular momentum respectively, as well as our control input consisting of the angular acceleration of the propellers, denoted by $\omega_r \in \R^3$.\footnote{We denote by $\S^n \subset \R^{n+1}$ the unit $n$-dimensional sphere inside $\R^{n+1}$.} The differential equation governing the dynamics is given by 
\begin{align}
    \label{quadcopter_ode}
    \dot{p} &= v \nonumber \\
    \dot{q} &= \frac{1}{2}E^T \Omega \nonumber \\
    \dot{v} &= \frac{1}{m}RF - g\mathbf{1}_z \\
    \dot{\Omega} &= J^{-1}( T + \Omega \times J \Omega ) \nonumber \\
    \dot{\omega} &= \omega_r, \nonumber
\end{align}
where 
\begin{align*}
    F &= \sum_{k=1}^4 \frac{1}{2}\rho A C_l \omega_k^2 \mathbf{1}_z \\
    T &= \begin{pmatrix} T_1 & T_2 & T_3 \end{pmatrix}^T,
\end{align*}
with
\begin{align*}
    T_1 &= \frac{AC_lL \rho (\omega_2^2 - \omega_4^2)}{2} \\
    T_2 &= \frac{AC_lL \rho (\omega_1^2 - \omega_3^2)}{2} \\
    T_3 &= \frac{AC_lL \rho (\omega_1^2 -\omega_2^2 + \omega_3^2 - \omega_4^2)}{2}.
\end{align*}

\begin{figure}[H]
\begin{center}
    \begin{tabular}{|c|c|l|}
    \hline
    Parameter & Value & Description \\
    \hline
    $\rho$ & 1.23 kg/m$^3$ & Air density\\
    $A$  & 0.1m$^2$ & Propeller area \\
    $C_l$ & 0.25 & Lift coefficient \\
    $C_d$ & $0.75$ & Drag coefficient \\
    $m$ & $10$kg & Quadrotor mass \\
    $g$ & $9.81$m/s$^2$ & Gravitational acceleration \\
    $J_1$ & $0.25$kg$\cdot$m/s$^2$ & First link length \\
    $J_2$ & $0.25$kg$\cdot$m/s$^2$ & Second link length \\
    $J_3$ & $1.0$kg$\cdot$m/s$^2$ & Third link length \\
    \hline
    \end{tabular}
\end{center}
\caption{Parameter values for quadrotor model}
\end{figure}

The objective function we use for our MPC controller is given by 
\begin{align*}
&(u - u_{ref})^T \cdot \diag(0.1) \cdot (u - u_{ref}) \\
+&p^T \cdot \diag(1.0) \cdot p + v^T \cdot \diag(0.1) \cdot v \\
+&(q - q_{ref})^T \cdot \diag(100.0) \cdot (q - q_{ref}) \\
+ &\Omega^T \cdot \diag(0.1) \cdot \Omega
\end{align*}
with
\begin{align*}
    u_{ref} &= 40\cdot \mathbf{1}_4 \\
    q_{ref} &= \begin{pmatrix}  1 \\ 0 \\ 0 \\ 0\end{pmatrix}
\end{align*}

As in the toy example above the FORCESPRO solver generated for this control problem takes 2 inputs: The initial condition for the differential equation (i.e. the current state of our dynamical system) as well as an initial guess for the optimization algorithm.

\begin{rem}
    The simulations in this example were performed on a desktop PC with a 12 (i7-8700M) core Intel CPU, 3.20 GHz, 12-13MB cache, running 64-bit Ubuntu 18.04.5 LTS.
\end{rem}

\subsubsection{The simulation model}
\label{quadcopter_simulation_model}
Writing $(p,v,q,\Omega)$ for the state of our system I will use the following statespace
\begin{equation}
    \label{quadcopter_statespace}
    \X = [-50, 50]^3 \times [-50, 50]^3 \times \S^3 \times [-5, 5]^3,
\end{equation}
meaning this case falls under \ref{assumption1}. The way we sample from the uniform distribution on $\S^3$ is by using acceptance-rejection sampling: We sample uniformly in $[-1,1]^4$ and accept only vectors $\bar{q}$ meeting the condition
\[
    10^{-6} < |\bar{q}| < 1,
\]
for which we then set $q = \frac{\bar{q}}{|\bar{q}|}$. 
For the RWM and the statistical models we used the following quantities
\begin{align*}
    r_p &= \begin{pmatrix} 0.1 & 0.1 & 0.1 & 0.1 \end{pmatrix} \\
    r_{RWM} &= \begin{pmatrix} 7.0 & 7.0 & 0.5 & 1.0 \end{pmatrix},
\end{align*}
where on the $\S^3$-component we simply perturb $q$ with a vector uniformly distributed in $[-0.1, 0.1]^4$ (respectively $[-0.5, 0.5]^4$) and then normalize.

\subsubsection{Convergence diagnostics}
\label{quadrotor_converence_diagnostics}
As a means of analysing the extend to which the underlying RWM Markov chain has converged we proceed with the strategy outlined in section \ref{conv_diag_sec}. For simplicity we use five pilot functions $p_l:\X \to \R$ which are linear combinations of trigonometric $\cos/\sin$ applied to each component of $x\in \X$. The coefficients in the linear combinations are chosen uniformly at random in the interval $(-10, 10)$. 

By running crude Monte Carlo simulation with a total of $1.3\cdot 10^9$ samples uniformly distributed in $\X$ I recorded a true sample $x$ of $1437$ independent realizations from $U|_{ \{ \F = 0 \} }$. This yields the following true/honest means:
\begin{align*}
    \overline{p}_1(x) &= -5.0705, \quad \overline{p}_2(x) = 10.6891, \quad \overline{p}_3(x) = -7.7518, \\ \overline{p}_4(x) &= -4.2142, \quad \overline{p}_5(x) = 13.1102
\end{align*}
Running an initial crude Monte Carlo simulation with $3\cdot 10^8$ samples yielded $323$ true initial conditions for the batched version of the RWM. Running each trajectory $10^5$ (yielding a total of $323\cdot 10^5$ samples) resulted in the following pilot means:
\begin{align*}
    \overline{p}_1(y) &= -5.1373, \quad \overline{p}_2(y) = 10.8220, \quad \overline{p}_3(y) = -7.7082, \\
    \overline{p}_4(y) &= -3.9662, \quad \overline{p}_5(y) = 13.6644
\end{align*}
While it is of course always preferable to use more samples, this empirical verification of converges of the batched RWM is good enough that we use the same number of samples for our mean and variance estimates for our confidence interval for $\P(\F(X_1)=0 | \F(X_0)=0)$.

\subsubsection{Latency budget model results}
For the latency budget model we will think of our FORCESPRO solver as a map $\X \to \U$
by feeding it the input given by the initial state $x$ and initial guess $(30\cdot1_{\R^4}^T, x^T,\ldots , 30\cdot 1_{\R^4}^T, x^T)$ at the point $x\in \X$.
We will use confidence level $\alpha = 1-10^{-6}$ for all confidence intervals.
Concerning the estimate for $\P(A_1)$ we use $3\cdot 10^8$ independent uniform samples in $\X$ to compute the maximum likelihood estimates
\[
\E(Z_1) \approx 1.17\cdot 10^{-6}, \quad \V(Z_1) \approx 1.17\cdot 10^{-6}.
\]
Hence, the $\alpha$-confidence upper bound for $\P(A_1)$ provided by the CLT is given by 
\begin{equation}
\label{quadrotor_x1_confidence_bound}
1.48\cdot 10^{-6}.
\end{equation}
For the RWM we follow the sample sizes we outlined in section \ref{quadrotor_converence_diagnostics}. Hence, based on a crude Monte Carlo simulation using $3\cdot 10^8$ samples we found $351$ initial conditions for our RWM simulation. We ran each trajectory in the simulation for $10^{5}$ steps, yielding:
\[
z_2 \approx 5.31 \cdot 10^{-5}, \quad v_2 \approx = 5.15\cdot 10^{-6}.
\]
Applying the (\ref{student_t_confidence_interval}) the we obtain the following $\alpha$-confidence interval for $\P(A_2 \ |\ A_1)$:
\[
(2.74 \cdot 10^{-5}, 7.87 \cdot 10^{-5}).
\]
Hence, assuming $N=2$ we get the following $\alpha$-confidence upper bound for the probability of the failure event for our quadrotor example:
\[
1.48\cdot 10^{-6} \cdot 7.87 \cdot 10^{-5} = 1.16 \cdot 10^{-10}.
\]

\subsubsection{Concurrent design results}
We will think of our FORCESPRO solver as a map 
\[
\X \times \X \to \U
\]
by feeding to it the system state $x$ and initial guess $(30\cdot1_{\R^4}^T, y^T,\ldots , 30\cdot 1_{\R^4}^T, y^T)$ at the point $(x,y) \in \X \times \X$. This allows us to set up a system as in Figure \ref{parallel_schematic} by simply perturbing the second component (as explained in (b) on page \pageref{parallel_schematic}).

We will use the same confidence level $\alpha$ here as in the above section. The estimator for $\P( A_1 )$ is the same as above, so the statistical upper bound is given by (\ref{quadrotor_x1_confidence_bound}). For estimating $\P (A_2\ |\ A_1)$ we used the $323$ failure samples found in the estimation of $\P( A_1 )$ as initial conditions to run MCMC simulations of $10^5$ samples each to obtain
\[
z_2 = 1.09 \cdot 10^{-5}, \quad v_2 = 4.10 \cdot 10^{-10}
\]
which (using (\ref{student_t_confidence_interval})) yields the following $\alpha$-confidence interval for $\P (A_2\ |\ A_1)$:
\[
(5.28 \cdot 10^{-6}, 1.65 \cdot 10^{-5}).
\]
This computation produced $352$ failure samples which we used as initial conditions to run MCMC simulations of $10^5$ samples each to obtain 
\[
z_3 = 4.56 \cdot 10^{-5}, \quad v_3 = 2.59 \cdot 10^{-9}.
\]
Using (\ref{student_t_confidence_interval}) this yields the following $\alpha$-confidence interval for $\P(A_3 \ |\ A_1 \cap A_2)$:
\[
(3.21 \cdot 10^{-5}, 5.91 \cdot 10^{-5})
\]
Hence, the statistical upper bound we get for our failure event is 
\[
\P ( \cap_{k} A_k ) \leq 1.48\cdot 10^{-6} \cdot 1.65 \cdot 10^{-5} \cdot 5.91 \cdot 10^{-5} = 1.44^{-15}.
\]

\section{Conclusion}
We have presented a framework for analysing the reliability of real-time control systems. The central idea is a decomposition technique allowing one to apply the \emph{subset simulation} technique to estimate the failure probability. The decomposition exploits the structure of two statistical models: One model for simulating the interaction between the controller and the system it controls and one model for designing a composite control system. This technique renders extremely small failure probabilities computationally feasible and system reliability can be verified. Very small failure probabilities were verified for NMPC control algorithms in challenging examples.

\hspace{0.25cm}
\subsubsection*{Acknowledgement}
A big thanks goes to Joachim Ferreau for a lot of help and valuable discussions concerning this topic. A special thank goes to Roberto Barbieri who has consistently provided good questions which helped guide this research. In particular his explanation of how cyclic redundancy checks are used to guarantee small probabilities of corruption non-detection in ethernet communication was crucial in developing the ideas presented here. Thanks go to the FORCESPRO team at Embotech AG for creating a good working atmosphere where ideas can be discussed and are allowed to grow. Such an environment is crucial for innovation to have a chance to happen. Last but certainly not least we very much appreciate the help provided by Carlo Alberto Pasccuci and Aldo Zgraggen in preparing a talk based on this paper for a ESA/CNES/DLR/NASA V\&V for G\&C workshop.

\appendix
\section{A glimpse of Monte Carlo}
\label{MCappendix}
We hope that the techniques presented in this paper are of interest to people in the control community, who might not be well-versed in Monte Carlo simulation theory. For this reason we here collect a few pointers to the literature which should help guide the rigor-minded reader towards more complete presentations of the basics. Explaining how these results come about or even giving any real insight is certainly outside the scope of this paper. 

The Monte Carlo simulation technique was conceived of\footnote{At least in its modern formulation.} under the Manhatten project in order to compute integrals related to nuclear devices. Later it found great use in nuclear technology and has only caught on to become a central computational tool across many scientific disciplines later, perhaps due to the often extremely demanding computational requirements \cite[Introduction]{Zio14}. The fundamental idea behind Monte Carlo simulation is to use the Central Limit Theorem (CLT) and the computational powers of modern computers to yield statistical information about the world we live in. A special case of the CLT is the following

\begin{thm}[Central Limit Theorem]
\label{clt_theorem}
Let $(Z_k)_{k=1}^{\infty}$ be a sequence of independent, identically distributed Bernoulli random variables with $p:=\P(Z_1=1)\in (0,1)$ and $\P(Z_1=0)=1-p$. Then the means $\overline{Z}_N:= \frac{1}{N}\sum_{k=1}^NZ_k$ satisfy
\[
\sqrt{N}( \overline{Z}_N - p ) \stackrel{\mathcal{D}}{\to} \mathcal{N}(0,\sigma^2),
\]
where $\mathcal{N}(0,\sigma^2)$ denotes the a normal distribution with mean $0$ and variance $\sigma^2=p(1-p)$.
\end{thm}
Typically, one sets up an experiment such that the quantity of interest is the mean $p = \E(Z_1)$ of a random variable $Z_1$. A statistician will then want to give an answer for $p = \E(Z_1)$ in terms of confidence intervals. The convergence in the CLT\footnote{The technical term for $\stackrel{\mathcal{D}}{\to}$ is "convergence in distribution".} in particular means that an approximate $\alpha$-confidence interval for $\E(Z_1)$ can be obtained by computing the $\alpha$-confidence interval given by $\mathcal{N}(0,\sigma^2)$ around the estimator $\overline{Z}_N$. Such a confidence interval has the explicit expression
\[
\overline{z}_N \pm \frac{t_{1-\alpha/2}\sigma}{\sqrt{N}},
\]
where $\overline{z}_N$ is the realization of $\overline{Z}_N$ and $t_{1-\alpha/2}$ is the $1-\frac{\alpha}{2}$ quantile of $\mathcal{N}(0,1)$. Computing this quantile is easily done using modern statistical software packages. As noted above $\sigma =\sqrt{p(1-p)}\approx \sqrt{p}$ if $p$ is small, so it can be seen from the above expression that $N \approx \frac{1}{p}$ samples are required to get $\frac{t_{1-\alpha/2}\sigma}{\sqrt{N}} \approx p$.

Typically one uses the Maximum Likelihood Estimate 
\[
\hat{\sigma}^2 := \frac{1}{N-1} \sum_{k=1}^N (z_k - \overline{z}_N)^2
\]
estimate for $\sigma^2$. Welford's algorithm computes this quantity in a numerically stable way "on the fly" \cite[Section 2.3]{Owen13}. See \cite[Chapter III]{AsmussenGlynn07} for an in depth discussion of this approach to computing approximate confidence intervals. 
Note that in the special case of I.I.D. Bernoulli variables presented in the above version of the CLT the sum $N\overline{Z}_N$ follows a $B(N,p)$ (Binomial) distribution so the cumulative distribution function of $B(N,p)$ can be used to compute an exact confidence interval for $\E(Z_1)$, see \cite[Chapter 2, Section 4]{Owen13}. 
The reason for presenting the CLT is that it is much more general and flexible.

The main simulation technique used in this paper is not the basic CLT but instead a Markov Chain Monte Carlo (MCMC) simulation technique known as the Metropolis-Hastings algorithm \cite{Metropolis1953, AsmussenGlynn07}. Perhaps the best introduction to MCMC is \cite{MCMCIntrobrooksGeyer}, but see also \cite{AsmussenGlynn07}, \cite{AuWang2014} or the original \cite{Metropolis1953}. This algorithm can be used if one is interested in sampling according to a distribution $f \cdot P$ which is absolutely continuous with respect to a probability measure $P$ from which one can sample (e.g. $P$ is the uniform distribution $U$ on $\X$ in section \ref{model_section}). A crucial feature of the algorithm used above is that it in fact suffices to know an unnormalized version $f_u$ of $f$.\footnote{I.e. a function $f_u$ such that $f = \frac{f_u}{\int f_u dP}$.} Given such $f_u$ and an initial value $x_0$ the algorithm determines $x_{k+1}$ from $x_k$ by sampling $q \sim Q(x_k)$ according to a \emph{proposal distribution} $Q(x_k)$ and setting 
\[
x_{k+1} = 
\begin{cases}
q, & \text{with probability} \ p:=\min \left( 1, \frac{f_u(x_{k+1})Q_{x_{k+1}}( x_k)}{f_u(x_{k})Q_{x_k}(x_{k+1})} \right) \\
x_k, & \text{with probability} \ 1-p, 
\end{cases}
\]
where $Q_x$ is the the density of the measure $Q(x)$ with respect to $P$. Under very general conditions the sequence $(x_k)_k$ will converge to a realization of a Markov chain $(X_k)_{k=1}^{\infty}$ which has $f \cdot P$ as a stationary distribution (see \cite{MCMCIntrobrooksGeyer}). Note that one cannot expect the $x_k$s to be independent, but there is a Markov Chain version of the CLT \cite[Chapter 11, Section 12]{Owen13}, saying that (if $N$ is large) an approximate $\alpha$-confidence interval for an integral of the form $\int F d(f\cdot P)$ is given by
\[
    \overline{F}_N \pm \frac{t_{1-\alpha/2}\sigma}{\sqrt{N}}
\]
where $\overline{F}_N = \frac{1}{N}\sum_{k=1}^N F(x_k)$ and 
\[
\sigma^2 = \V(X_1) +2\sum_{k=1}^{\infty} \mathbb{Cov}(X_1, X_{1+k}).
\]
The approach we take to estimating this $\sigma$ in this paper is by using a batch approach \cite[Chapter IV, Section 5]{AsmussenGlynn07}.

\bibliographystyle{plain}
\bibliography{BIB}
\end{document}